\documentclass[12pt,reqno]{amsart}
\usepackage{amsfonts}
\usepackage{amssymb,color}
\usepackage{amssymb}

\usepackage[colorlinks=true]{hyperref}
\hypersetup{urlcolor=blue, citecolor=blue}

\usepackage[margin=2.6cm, a4paper]{geometry}

\newtheorem{theo}{Theorem}[section]
\newtheorem{lemm}[theo]{Lemma}

\numberwithin{equation}{section}

\newcommand{\bal}{\begin{align}}
\newcommand{\bbal}{\begin{align*}}
\newcommand{\beq}{\begin{equation}}
\newcommand{\eeq}{\end{equation}}
\newcommand{\bca}{\begin{cases}}
\newcommand{\eca}{\end{cases}}

\newcommand{\na}{\nabla}
\newcommand{\De}{\Delta}

\newcommand{\cd}{\cdot}

\newcommand{\dd}{\mathrm{d}}

\newcommand{\R}{\mathbb{R}}

\begin{document}


\title[The continuous dependence for the Hal-MHD equations]{The continuous dependence for the Hal-MHD equations with fractional magnetic diffusion}

\author[D. Zhong]{Dingxing Zhong}
\address{School of Mathematics and Computer Sciences, Gannan Normal University, Ganzhou 341000, China}
\email{zhongdingxing678@sina.com}

\author[J. Li]{Jinlu Li}
\address{School of Mathematics and Computer Sciences, Gannan Normal University, Ganzhou 341000, China}
\email{lijinlu@gnnu.cn}

\author[X. Wu]{Xing Wu}
\address{College of Information and Management Science, Henan Agricultural University, Zhengzhou 450002, China}
\email{ny2008wx@163.com}

\begin{abstract}
In this paper we show that the solutions to the incompressible Hall-MHD system with fractional magnetic diffusion depend continuously on the initial data in $H^s(\mathbb{R}^d)$, $s>1+\frac{d}{2}$.
\end{abstract}

\maketitle

\section{Introduction and main result}

In this paper, we are concerned with the study on the Cauchy problem for the incompressible Hall-MHD system with fractional magnetic diffusion:
\beq\label{MHD1}\begin{cases}
\partial_tu+u\cdot\nabla u+\nabla P=b\cdot \nabla b, \\
\partial_tb+u\cdot \nabla b+\nabla \times((\nabla \times b)\times b)+(-\De)^{\alpha} b=b\cdot\nabla u,\\
\mathrm{div} u=\mathrm{div} b=0,\quad (u,b)|_{t=0}=(u_0,b_0),
\end{cases}\eeq
where $u$ and $b$ represent the flow velocity vector and the magnetic field vector respectively, $P$ is a scalar pressure, while $u_0(x)$ and $b_0(x)$ are the given initial velocity and initial magnetic field with $\nabla\cdot u_0=\nabla\cdot b_0=0$. The Hall term $\nabla\times((\nabla \times b)\times b)$ plays an important role in magnetic reconnection which is happening in the case of large magnetic shear. The fractional Laplacian operator $(-\Delta)^\alpha$ is defined in terms of the Fourier transform by $\widehat{(-\Delta)^\alpha}f(\xi)=|\xi|^{2\alpha}\hat{f}(\xi)$.

The applications of the Hall-MHD system cover a very wide range of physical objects, such as, in the magnetic reconnection in space plasmas, the star formation, neutron stars and geo-dynamos.

Recently, there are many researches on the standard Hall-MHD equations with  $-\Delta u$ and $-\Delta b$ , concerning global weak solutions \cite{1Chae 2014}, local and global (small) strong solutions \cite{2Chae 2014, Wu 2018}, and the large time behavior of weak and strong solutions \cite{Wan 2015, Chae 2013, 1Weng 2016, 2Weng 2016}.
For the system \eqref{MHD1},  Chae, Wan and Wu \cite{3Chae 2014}  proved the local existence and uniqueness of the solution to the Hall-MHD equations with only a fractional Laplacian magnetic diffusion $(-\Delta)^\alpha b$ in the space $H^s(\mathbb{R}^d)$ for $s>1+\frac{d}{2}$ and $\alpha>\frac{1}{2}$. To the knowledge of the author, there are few results on continuous dependence upon the initial data for the Hall-MHD system. Therefore, the purpose of our paper is to discuss the continuous dependence of solutions on the initial data in \eqref{MHD1} and give a precise estimate.

Our argument can state as follows:
\begin{theo}\label{th1}
Let $d\geq 2$, $\alpha\geq 1$ and $s>1+\frac d2$. Suppose that $u^n_0\in H^s(\R^d)$ goes to $u_0\in H^s(\R^d)$ in $H^s(\R^d)$ and $b^n_0\in H^s(\R^d)$ goes to $b_0\in H^s(\R^d)$ in $H^s(\R^d)$ when $n$ goes to infinity. Then there exists a positive $T>0$ independent of $n$ such that  $(u^n,b^n)\in \mathcal{C}([0,T];H^s(\R^d))$ be the solution of
\bal\label{MHD2}\begin{cases}
\partial_tu^n+u^n\cdot\nabla u^n+\nabla P_n=b^n\cdot \nabla b^n, \\
\partial_tb^n+u^n\cdot \nabla b^n+\nabla \times((\nabla \times b^n)\times b^n)+(-\De)^{\alpha} b^n=b^n\cdot\nabla u^n,\\
\mathrm{div} u^n=\mathrm{div} b^n=0,\quad (u^n,b^n)|_{t=0}=(u^n_0,b^n_0),
\end{cases}\end{align}
and $(u,b)\in \mathcal{C}([0,T];H^s(\R^d))$ be the solution of \eqref{MHD1} with initial data $(u_0,b_0)$. Moreover, there holds
\bbal
\lim_{n\rightarrow \infty}\big(||u^n-u||_{L^\infty_T(H^s(\R^d))}+||b^n-b||_{L^\infty_T(H^s(\R^d))}\big)=0.
\end{align*}
\end{theo}

\noindent\textbf{Notations.} Given a Banach space $X$, we denote its norm by $\|\cdot\|_{X}$. Since all spaces of functions are over $\mathbb{R}^d$, for simplicity, we drop  $\mathbb{R}^d$ in our notations of function spaces if there is no ambiguity. The symbol $A\lesssim B$ denotes that there exists a constant $\bar{c}_0>0$ independent of $A$ and $B$, such that $A\leq \bar{c}_0 B$. We denote by $\{c_j\}_{j\geq -1}$ a sequence such that $||c_j||_{\ell^2}\leq 1$.

\section{Preliminaries}

In this section we collect some preliminary definitions and lemmas. For more details we refer the readers to \cite{B.C.D}. \\

Let $\chi: {\mathbb R}^d\to [0, 1]$ be a radial, non-negative,
smooth and radially decreasing function which is supported in $\mathcal{B}\triangleq \{\xi:|\xi|\leq \frac43\}$ and
$\chi\equiv 1$ for $|\xi|\leq \frac54$. Let $\varphi(\xi)=\chi(\frac{\xi}{2})-\chi(\xi)$. Then $\varphi$ is supported in the ring $\mathcal{C}\triangleq \{\xi\in\mathbb{R}^d:\frac 3 4\leq|\xi|\leq \frac 8 3\}$.
For $u \in \mathcal{S}'$, $q\in {\mathbb Z}$, we define the Littlewood-Paley operators: $\dot{\Delta}_q{u}=\mathcal{F}^{-1}(\varphi(2^{-q}\cdot)\mathcal{F}u)$, ${\Delta}_q{u}=\dot{\Delta}_q{u}$ for $q\geq 0$, ${\Delta}_q{u}=0$ for $q\leq -2$ and $\Delta_{-1}u=\mathcal{F}^{-1}(\chi \mathcal{F}u)$, and $S_q{u}=\mathcal{F}^{-1}\big(\chi(2^{-q}\xi)\mathcal{F}u\big)$.
Here we use ${\mathcal{F}}(f)$ or $\widehat{f}$ to denote
the Fourier transform of $f$.

\begin{lemm}\cite{B.C.D}\label{le1}
Let $s>0$. Then there exists a constant $C$, depending only on $d,s$, such that for all $f,g\in H^s$,
\[\|fg\|_{H^{s}}\leq C\big(\|f\|_{L^\infty}||g||_{H^s}+||f||_{H^s}\|g\|_{L^\infty}\big).\]
\end{lemm}

\begin{lemm}\cite{L.Y}\label{le2}
Let $\sigma\in\mathbb{R}$ and $v$ be a vector field over $\mathbb{R}^d$. Assume that $\sigma>-d\min\{1-\frac{1}{p},\frac{1}{p}\}$. Define $R_j=[v\cdot\nabla,\Delta_j]f$. There a constant $C=C(p,\sigma,d)$ such that \\
\begin{equation*}
\big|\big|(2^{j\sigma}||R_j||_{L^2})_{j\geq-1}\big|\big|_{\ell^r}\leq
\begin{cases}
 C||\nabla v||_{H^{\frac d2}\cap L^\infty}||f||_{H^\sigma},\ \mathrm{if} \quad \sigma<1+\frac{d}{2},\\
 C||\nabla v||_{H^{\frac d2+1}}||f||_{H^{\sigma}},\quad \quad \mathrm{if} \quad \sigma=1+\frac{d}{2},\\
 C||\nabla v||_{H^{\sigma-1}}||f||_{H^\sigma}, \quad \mathrm{if} \quad \sigma>1+\frac{d}{2}.
\end{cases}\end{equation*}
Moreover, if $\sigma>0$, we have
\begin{equation*}
\big|\big|(2^{j\sigma}||R_j||_{L^2})_{j\geq-1}\big|\big|_{\ell^r}\leq C(||\na v||_{L^\infty}||f||_{H^\sigma}+||\na f||_{L^\infty}||\na v||_{H^{\sigma-1}}).
\end{equation*}
\end{lemm}

\section{Some useful estimates}

In this section, motivated by \cite{G.L.Y}, we will establish some useful estimates for smooth solutions of \eqref{MHD1}, which is the key component in the proof of Theorems \ref{th1}. The estimates can be stated as follows.

\begin{lemm}\label{le-0}
Let $d\geq 2$ and $s>1+\frac d2$. Suppose that $(u,b)\in \mathcal{C}([0,T];H^{s})$ is the solution of  \eqref{MHD1} with initial data $(u_0,b_0)$. Then, we have for all $t\in[0,T]$,
\bbal
||u(t)||^2_{H^s}+||b(t)||^2_{H^s}\leq (||u_0||^2_{H^s}+||b_0||^2_{H^s})e^{C\int^T_0(||u(t)||_{C^{0,1}}+||b(t)||_{C^{0,1}}+||b(t)||^2_{C^{0,1}})\dd t}.
\end{align*}
\end{lemm}
\begin{proof}
Now, we apply $\De_j$ to \eqref{MHD1}, and take the inner product with $(\De_ju,\De_jb)$ and integrate by parts to have
\bal\label{eq-li}
\frac12\frac{\dd}{\dd t}(||\De_ju||^2_{L^2}+||\De_jb||^2_{L^2})+C_02^{2j\alpha}||\De_j b||^2_{L^2}\leq \sum^5_{i=1}I_i,
\end{align}
where
\bbal
&I_1=-\int_{\R^d}[\De_j,u\cdot \na]u\cd \De_ju\ \dd x, \quad \quad I_2=-\int_{\R^d}[\De_j,u\cdot \na] b\cd \De_j b\ \dd x,\\
&I_3=\int_{\R^d}[\De_j,b\cdot \na] b\cd \De_j u\ \dd x, \quad \quad I_4=\int_{\R^d}[\De_j,b\cdot \na] u\cd \De_j b\ \dd x,\\
&I_5=\int_{\R^d}[\De_j,b\times](\nabla \times b)\cdot (\nabla\times \De_j b)\ \dd x.
\end{align*}
According to Lemmas \ref{le1}-\ref{le2}, it is easy to estimate
\bbal
&|I_1|\lesssim ||[\De_j,u\cdot \na]u||_{L^2}||\De_ju||_{L^2}\lesssim 2^{-2js}c^2_j||\na u||_{L^\infty}|| u||^2_{H^{s}},\\
&|I_2|\lesssim ||[\De_j,u\cdot \na]b||_{L^2}||\De_j b||_{L^2}\lesssim 2^{-2js}c^2_j\big(||\na u||_{L^\infty}||b||^2_{H^{s}}+||\na b||_{L^\infty}||u||_{H^{s}}||b||_{H^{s}}\big),\\
&|I_3|\lesssim ||[\De_j,b\cdot \na] b||_{L^2}||\De_j u||_{L^2}\lesssim 2^{-2js}c^2_j||\na b||_{L^\infty}|| u||_{H^{s}}|| b||_{H^{s}},\\
&|I_4|\lesssim ||[\De_j,b\cdot \na] u||_{L^2}||\De_j b||_{L^2}\lesssim 2^{-2js}c^2_j\big(||\na b||_{L^\infty}||u||_{H^{s}}|| b||_{H^{s}}+||\na u||_{L^\infty}||b||^2_{H^{s}}\big).
\end{align*}
By Lemma \ref{le2}, we can gain that for $j\geq 0$,
\bbal
|I_5|&\lesssim ||[\De_j,b\times](\nabla \times b)||_{L^2}\cdot(2^j||\De_jb||_{L^2})
\\&\leq C||[\De_j,b\times](\nabla \times b)||^2_{L^2}+\frac{C_0}{8}2^{2j\alpha}||\De_j b||^2_{L^2}
\\&\leq C2^{-2j(s-1)}c^2_j||\nabla b||^{2}_{L^\infty}||b||^2_{H^{s}}+\frac{C_0}{8}2^{2j\alpha}||\De_jb||^2_{L^2}.
\end{align*}
If $j=-1$, it is easy to deduce that
\bbal
|I_5|\lesssim ||b||_{C^{0,1}}||b||^2_{H^{s}}.
\end{align*}
Integrating \eqref{eq-li} over $[0,t]$, multiplying the inequality above by $2^{2j{s}}$ and summing over $j\geq -1$, we have
\bbal\label{eq3}
&||u(t)||^2_{H^{s}}+||b(t)||^2_{H^{s}}\leq ||u_0||^2_{H^{s}}+||b_0||^2_{H^{s}}\nonumber\\& \quad +\int^t_0(|| u||^2_{H^{s}}+|| b||^2_{H^{s}})(||u(\tau)||_{C^{0,1}}+||b(\tau)||_{C^{0,1}}+||b(\tau)||^2_{C^{0,1}})\dd \tau .
\end{align*}
This along with the Gronwall inequality yields the result of this lemma.
\end{proof}

\begin{lemm}\label{le-1}
Let $d\geq 2$ and $s>1+\frac d2$. Suppose that $(u^1,b^1)$ and $(u^2,b^2)$ are two smooth solutions of \eqref{MHD1} with initial data $(u^1_0,b^1_0)$ and $(u^2_0,b^2_0)$ respectively. Denote $\delta u=u^1-u^2$ and $\delta b=b^1-b^2$. Then, we have for all $t\in[0,T]$,
\bbal
||\delta u(t)||^2_{H^{s-1}}+||\delta b(t)||^2_{H^{s-1}}\leq (||\delta u_0||^2_{H^{s-1}}+||\delta b_0||^2_{H^{s-1}})e^{\mathrm{A}(t)},
\end{align*}
and
\begin{align*}
||\delta u(t)||^2_{H^{s}}&+||\delta b(t)||^2_{H^{s}}\leq \Big(||\delta u_0||^2_{H^{s}}+||\delta b_0||^2_{H^{s}}\\& \quad +C\int^t_0(||u^2||^2_{H^{s+1}}+||b^2||^2_{H^{s+1}})(||\delta u||^2_{H^{s-1}}+||\delta b||^2_{H^{s-1}}) \dd \tau\Big)e^{\mathrm{A}(t)},
\end{align*}
with
\bbal
\mathrm{A}(t)=C\int^t_0(1+||u^1||_{H^s}+||u^2||_{H^s}+||b^1||_{H^s}+||b^2||_{H^s}+||b^1||^2_{H^s}+||b^2||^2_{H^s})\dd \tau.
\end{align*}
\end{lemm}
\begin{proof}
It is easy to show that
\beq\label{MHD1-1}\bca
\partial_t\delta u+u^1\cdot\nabla \delta u+\delta u\cd\nabla u^2+\nabla \bar{P}=b^1\cdot \nabla \delta b+\delta b\cd \na b^2, \\
\partial_t\delta b+u^1\cdot \nabla \delta b+\delta u\cd \na b^2+(-\De)^{\alpha} \delta b+\nabla\times((\nabla\times \delta b)\times b^1)\\ \qquad
+\nabla\times((\nabla\times b^2)\times \delta b)
=b^1\cdot\nabla \delta u+\delta b\cd \na u^2,\\
\mathrm{div}\delta u=\mathrm{div}\delta b=0,\quad (\delta u,\delta b)|_{t=0}=(\delta u_0,\delta b_0).
\eca\eeq
Now, we apply $\De_j$ to \eqref{MHD1-1}, and take the inner product with $(\De_j\delta u,\De_j\delta b)$ and integrate by parts to have
\bal\label{eq1}
\frac12\frac{\dd}{\dd t}(||\De_j\delta u||^2_{L^2}+||\De_j\delta b||^2_{L^2})+C_02^{2j\alpha}||\De_j\delta b||^2_{L^2}\leq \sum^8_{i=1}J_i,
\end{align}
where
\bbal
&J_1=-\int_{\R^d}[\De_j,u^1\cdot \na]\delta u\cd \De_j\delta u\ \dd x, \quad \quad J_2=-\int_{\R^d}[\De_j,u^1\cdot \na]\delta b\cd \De_j\delta b\ \dd x,\\
&J_3=\int_{\R^d}[\De_j,b^1\cdot \na]\delta b\cd \De_j\delta u\ \dd x, \quad \quad J_4=\int_{\R^d}[\De_j,b^1\cdot \na]\delta u\cd \De_j\delta b\ \dd x,\\
&J_5=\int_{\R^d}\De_j(\delta b\cd \na b^2-\delta u\cd\nabla u^2)\De_j\delta u\ \dd x,\\
&J_6=\int_{\R^d}\De_j(\delta b\cd \na u^2-\delta u\cd\nabla b^2)\De_j\delta b\ \dd x,\\
&J_7=\int_{\R^d}[\De_j,b^1\times](\nabla \times \delta b)\cdot (\nabla\times \De_j\delta b)\ \dd x,\\
&J_8=-\int_{\R^d}\De_j((\nabla \times b^2)\times \delta b)\cdot (\nabla\times \De_j\delta b)\ \dd x.
\end{align*}
On the one hand, according to Lemmas \ref{le1}-\ref{le2}, it is easy to estimate
\begin{align}\label{eq-z1}\begin{split}
&|J_1|\lesssim ||[\De_j,u^1\cdot \na]\delta u||_{L^2}||\De_j\delta u||_{L^2}\lesssim 2^{-2j(s-1)}c^2_j||\na u^1||_{H^{s-1}}||\delta u||^2_{H^{s-1}},\\
&|J_2|\lesssim ||[\De_j,u^1\cdot \na]\delta b||_{L^2}||\De_j\delta b||_{L^2}\lesssim 2^{-2j(s-1)}c^2_j||\na u^1||_{H^{s-1}}||\delta b||^2_{H^{s-1}},\\
&|J_3|\lesssim ||[\De_j,b^1\cdot \na]\delta b||_{L^2}||\De_j\delta u||_{L^2}\lesssim 2^{-2j(s-1)}c^2_j||\na b^1||_{H^{s-1}}||\delta u||_{H^{s-1}}||\delta b||_{H^{s-1}},\\
&|J_4|\lesssim ||[\De_j,b^1\cdot \na]\delta u||_{L^2}||\De_j\delta b||_{L^2}\lesssim 2^{-2j(s-1)}c^2_j||\na b^1||_{H^{s-1}}||\delta u||_{H^{s-1}}||\delta b||_{H^{s-1}},\\
&|J_5|\lesssim 2^{-2j(s-1)}c^2_j(||b^2||_{H^{s}}||\delta u||_{H^{s-1}}||\delta b||_{H^{s-1}}+||u^2||_{H^s}||\delta u||^2_{H^{s-1}}),\\
&|J_6|\lesssim 2^{-2j(s-1)}c^2_j(||b^2||_{H^{s}}||\delta u||_{H^{s-1}}||\delta b||_{H^{s-1}}+||u^2||_{H^s}||\delta b||^2_{H^{s-1}}).
\end{split}\end{align}
By Lemma \ref{le2}, we obtain for $j\geq 0$,
\bal\label{eq-z2}\begin{split}
|J_7|&\lesssim ||[\De_j,b^1\times](\nabla \times \delta b)||_{L^2}\cdot(2^j||\De_j\delta b||_{L^2})
\\&\leq C||[\De_j,b^1\times](\nabla \times \delta b)||^2_{L^2}+\frac{C_0}{8}2^{2j\alpha}||\De_j\delta b||^2_{L^2}
\\&\leq C2^{-2j(s-1)}c^2_j||b^1||^{2}_{H^s}||\delta b||^2_{H^{s-1}}+\frac{C_0}{8}2^{2j\alpha}||\De_j\delta b||^2_{L^2}.
\end{split}\end{align}
By Lemma \ref{le1}, we have for $j\geq 0$,
\bal\label{eq-z3}\begin{split}
|J_8|&\lesssim ||\De_j((\nabla \times b^2)\times \delta b)||_{L^2}\cdot(2^j||\De_j\delta b||_{L^2})
\\&\leq C||\De_j((\nabla \times b^2)\times \delta b)||^2_{L^2}+\frac{C_0}{8}2^{2j\alpha}||\De_j\delta b||^2_{L^2}
\\&\leq C2^{-2j(s-1)}c^2_j||b^2||^{2}_{H^s}||\delta b||^2_{H^{s-1}}+\frac{C_0}{8}2^{2j\alpha}||\De_j\delta b||^2_{L^2}.
\end{split}\end{align}
If $j=-1$, it is easy to deduce that
\bal\label{eq-z4}
|J_7|+|J_8|\lesssim (||b^1||_{H^s}+||b^2||_{H^s})||\delta b||^2_{H^{s-1}}.
\end{align}
Integrating \eqref{eq1} over $[0,t]$, multiplying the inequality above by $2^{2j(s-1)}$ and summing over $j\geq -1$, we infer from \eqref{eq-z1}-\eqref{eq-z4} that
\bbal
||\delta u||^2_{H^{s-1}}+||\delta b||^2_{H^{s-1}}&\leq ||\delta u_0||^2_{H^{s-1}}+||\delta b_0||^2_{H^{s-1}}\nonumber\\& \quad +\int^t_0\mathrm{A}'(\tau)(||\delta u||^2_{H^{s-1}}+||\delta b||^2_{H^{s-1}})\dd \tau.
\end{align*}
This along with the Gronwall inequality yields the first part of the lemma. On the other hand, according to Lemmas \ref{le1}-\ref{le2}, it is easy to estimate
\bal\label{eq-z5}\begin{split}
&|J_1|\lesssim ||[\De_j,u^1\cdot \na]\delta u||_{L^2}||\De_j\delta u||_{L^2}\lesssim 2^{-2js}c^2_j||\na u^1||_{H^{s-1}}||\delta u||^2_{H^{s}},\\
&|J_2|\lesssim ||[\De_j,u^1\cdot \na]\delta b||_{L^2}||\De_j\delta b||_{L^2}\lesssim 2^{-2js}c^2_j||\na u^1||_{H^{s-1}}||\delta b||^2_{H^{s}},\\
&|J_3|\lesssim ||[\De_j,b^1\cdot \na]\delta b||_{L^2}||\De_j\delta u||_{L^2}\lesssim 2^{-2js}c^2_j||\na b^1||_{H^{s-1}}||\delta u||_{H^{s}}||\delta b||_{H^{s}},\\
&|J_4|\lesssim ||[\De_j,b^1\cdot \na]\delta u||_{L^2}||\De_j\delta b||_{L^2}\lesssim 2^{-2js}c^2_j||\na b^1||_{H^{s-1}}||\delta u||_{H^{s}}||\delta b||_{H^{s}},\\
&|J_5|\lesssim 2^{-2js}c^2_j(||b^2||_{H^{s}}||\delta u||_{H^{s}}||\delta b||_{H^{s}}+||u^2||_{H^s}||\delta u||^2_{H^{s}}\\& \quad \quad
\quad +||b^2||_{H^{s+1}}||\delta u||_{H^{s}}||\delta b||_{H^{s-1}}+||u^2||_{H^{s+1}}||\delta u||_{H^{s}}||\delta u||_{H^{s-1}}),\\
&|J_6|\lesssim 2^{-2js}c^2_j(||b^2||_{H^{s}}||\delta u||_{H^{s}}||\delta b||_{H^{s}}+||u^2||_{H^s}||\delta b||^2_{H^{s}}\\& \quad \quad
\quad +||b^2||_{H^{s+1}}||\delta u||_{H^{s-1}}||\delta b||_{H^{s}}+||u^2||_{H^{s+1}}||\delta b||_{H^{s}}||\delta b||_{H^{s-1}}).
\end{split}\end{align}
By Lemma \ref{le2}, we can gain for $j\geq 0$,
\bal\label{eq-z6}\begin{split}
|J_7|&\lesssim ||[\De_j,b^1\times](\nabla \times \delta b)||_{L^2}\cdot(2^j||\De_j\delta b||_{L^2})
\\&\leq C||[\De_j,b^1\times](\nabla \times \delta b)||^2_{L^2}+\frac{C_0}{8}2^{2j\alpha}||\De_j\delta b||^2_{L^2}
\\&\leq C2^{-2j(s-1)}c^2_j||b^1||^{2}_{H^s}||\delta b||^2_{H^{s}}+\frac{C_0}{8}2^{2j\alpha}||\De_j\delta b||^2_{L^2}.
\end{split}\end{align}
By Lemma \ref{le1}, we have for $j\geq 0$,
\bal\label{eq-z7}\begin{split}
|J_8|&\lesssim ||\De_j((\nabla \times b^2)\times \delta b)||_{L^2}\cdot(2^j||\De_j\delta b||_{L^2})
\\&\leq C||\De_j((\nabla \times b^2)\times \delta b)||^2_{L^2}+\frac{C_0}{8}2^{2j\alpha}||\De_j\delta b||^2_{L^2}
\\&\leq C2^{-2js}c^2_j\big(||b^2||^{2}_{H^{s+1}}||\delta b||^2_{H^{s-1}}+||b^2||^{2}_{H^{s}}||\delta b||^2_{H^s}\big)+\frac{C_0}{8}2^{2j\alpha}||\De_j\delta b||^2_{L^2}.
\end{split}\end{align}
If $j=-1$, by \eqref{eq-z4}, we have
\bal\label{eq-z8}
|J_7|+|J_8|\lesssim (||b^1||_{H^s}+||b^2||_{H^s})||\delta b||^2_{H^{s}}.
\end{align}
Integrating \eqref{eq1} over $[0,t]$, multiplying the inequality above by $2^{2j{s}}$ and summing over $j\geq -1$, we infer from \eqref{eq-z5}-\eqref{eq-z8} that
\bal\label{eq3}
||\delta u||^2_{H^{s}}+||\delta b||^2_{H^{s}}&\leq ||\delta u_0||^2_{H^{s}}+||\delta b_0||^2_{H^{s}}\nonumber\\& \quad +\int^t_0\mathrm{A}'(\tau)(||\delta u||^2_{H^{s}}+||\delta b||^2_{H^{s}})\dd \tau \nonumber
\\& \quad +C\int^t_0(||u^2||^2_{H^{s+1}}+||b^2||^2_{H^{s+1}})(||\delta u||^2_{H^{s-1}}+||\delta b||^2_{H^{s-1}}) \dd \tau.
\end{align}
This along with the Gronwall inequality yields the second part of the lemma. Therefore, we complete the proof of this lemma.
\end{proof}

\section{Proof of Theorem 1.1}
In this section, we will give the details of the proof for the Theorem 1.1.

\noindent\textbf{Proof of Theorem 1.1.} First, according to classical results, there exist a positive $T_n>0$ such that \eqref{MHD2} have a solution $(u^n,b^n)\in \mathcal{C}([0,T_n);H^s)$. Indeed, by Lemma \ref{le-0}, we have
\bbal
||u^n||^2_{H^s}+||b^n||^2_{H^s}\leq (||u^n_0||^2_{H^s}+||b^n_0||^2_{H^s})
e^{C\int^t_0(||u^n||_{C^{0,1}}+||b^n||_{C^{0,1}}+||b^n||^2_{C^{0,1}})\dd \tau}.
\end{align*}
Denote $R=\sup\limits_{n\geq 0}(||u^n_0||_{H^s}+||b^n_0||_{H^s})$. Therefore, by continuity arguments, there exists a positive $T=T(s,d,R)$ satisfying $T_n>T$ such that
\bbal
||u^n||^2_{L^\infty_T(H^s)}+||b^n||^2_{L^\infty_T(H^s)}\leq C(||u^n_0||^2_{H^s}+||b^n_0||^2_{H^s})\leq C.
\end{align*}
Moreover, for all $\gamma>s$, we have for all $t\in[0,T]$,
\bbal
||u^n(t)||^2_{H^\gamma}+||b^n(t)||^2_{H^\gamma}&\leq (||u^n_0||^2_{H^\gamma}+||b^n_0||^2_{H^\gamma})
e^{C\int^t_0(||u^n||_{C^{0,1}}+||b^n||_{C^{0,1}})\dd \tau}\\&\leq C(||u^n_0||^2_{H^\gamma}+||b^n_0||^2_{H^\gamma}).
\end{align*}
Let $(u^n_j,b^n_j)\in \mathcal{C}([0,T];H^s)$ be the solution of
\beq\begin{cases}
\partial_tu^n_j+u^n_j\cdot\nabla u^n_j+\nabla P_{n,j}=b^n_j\cdot \nabla b^n_j, \\
\partial_tb^n_j+u^n_j\cdot \nabla b^n_j+\nabla \times((\nabla \times b^n_j)\times b^n_j)+(-\De)^{\alpha} b^n_j=b^n_j\cdot\nabla u^n_j,\\
\mathrm{div} u^n_j=\mathrm{div} b^n_j=0,\quad (u^n_j,b^n_j)|_{t=0}=S_j(u^n_0,b^n_0).
\end{cases}\eeq
Then, according to Lemma \ref{le-1}, we have
\bbal
&||u^n_j-u^n||^2_{L^\infty_T(H^{s-1})}+||b^n_j-b^n||^2_{L^\infty_T(H^{s-1})}\\&
\leq  C(||(\mathrm{Id}-S_j)u^n_0||^2_{H^{s-1}}+||(\mathrm{Id}-S_j)b^n_0||^2_{H^{s-1}}),
\end{align*}
which along with the fact that $||u^n_j||_{L^\infty_T(H^{s+1})}+||b^n_j||_{L^\infty_T(H^{s+1})}\leq C2^j$ leads to
\bal\label{eq-w1}\begin{split}
& \quad \ ||u^n_j-u^n||^2_{L^\infty_T(H^{s})}+||b^n_j-b^n||^2_{L^\infty_T(H^{s})}\\&\leq C(||(\mathrm{Id}-S_j)u^n_0||^2_{H^{s}}+||(\mathrm{Id}-S_j)b^n_0||^2_{H^{s}}\\& \quad \ +\int^t_0(||u^n_j||^2_{H^{s+1}}+||b^n_j||^2_{H^{s+1}})(||u^n_j-u^n||^2_{H^{s-1}}+||b^n_j-b^n||^2_{H^{s-1}}) \dd \tau)
\\&\leq C(||(\mathrm{Id}-S_j)u^n_0||^2_{H^{s}}+||(\mathrm{Id}-S_j)b^n_0||^2_{H^{s}}
\\& \quad \ +2^{2j}||(\mathrm{Id}-S_j)u^n_0||^2_{H^{s-1}}+2^{2j}||(\mathrm{Id}-S_j)b^n_0||^2_{H^{s-1}})
\\& \leq C(||(\mathrm{Id}-S_j)u^n_0||^2_{H^{s}}+||(\mathrm{Id}-S_j)b^n_0||^2_{H^{s}}).
\end{split}\end{align}
Using Lemma \ref{le-1}, we can obtain
\bal\label{eq-w2}
||u^n_j-u_j||^2_{L^\infty_T(H^{s})}+||b^n_j-b_j||^2_{L^\infty_T(H^{s})}\leq C2^{2j}(||u^n_0-u_0||^2_{H^s}+||b^n_0-b_0||^2_{H^s}).
\end{align}
Therefore, combing \eqref{eq-w1}-\eqref{eq-w2}, we obtain
\bbal
& \quad \ ||u^n-u||^2_{L^\infty_T(H^s)}+||b^n-b||^2_{L^\infty_T(H^s)}\\&\leq ||u^n_j-u_j||^2_{L^\infty_T(H^s)}+||b^n_j-b_j||^2_{L^\infty_T(H^s)}
\\&\quad  +||u^n_j-u^n||^2_{L^\infty_T(H^s)}+||b^n_j-b^n||^2_{L^\infty_T(H^s)}
\\&\quad +||u_j-u||^2_{L^\infty_T(H^s)}+||b_j-b||^2_{L^\infty_T(H^s)}
\\& \leq  C(||(\mathrm{Id}-S_j)u^n_0||^2_{H^{s}}+||(\mathrm{Id}-S_j)b^n_0||^2_{H^{s}}+||(\mathrm{Id}-S_j)u_0||^2_{H^{s}}\\& \quad \ +||(\mathrm{Id}-S_j)b_0||^2_{H^{s}}+
2^{2j}||u^n_0-u_0||^2_{H^s}+2^{2j}||b^n_0-b_0||^2_{H^s})
\\& \leq  C(||(\mathrm{Id}-S_j)u_0||^2_{H^{s}}+||(\mathrm{Id}-S_j)b_0||^2_{H^{s}}\\& \quad \ +
2^{2j}||u^n_0-u_0||^2_{H^s}+2^{2j}||b^n_0-b_0||^2_{H^s}).
\end{align*}
This implies the result of Theorem 1.1. \\

\vspace*{1em}
\noindent\textbf{Acknowledgements.} This work was partially supported by NSFC (No. 11361004).
%\vspace*{1em}

\end{document}